\documentclass[preprint,12pt]{elsarticle}

\usepackage{graphicx}

\usepackage{amssymb}
\usepackage{amsthm}

\journal{ }

\begin{document}

\begin{frontmatter}



\title{Packing $(2^{k+1}-1)$-order perfect binary trees into (\emph{k}+1)-connected graph}


\author{Jia Zhao}
\address{Beijing Jiaotong University, Beijing, China, \{11111004@bjtu.edu.cn\}}

\author{Jianfeng Guan}
\address{Beijing University of Posts and Telecommunications, Beijing, China, \{jfguan@bupt.edu.cn\}}

\author{Changqiao Xu}
\address{Beijing University of Posts and Telecommunications, Beijing, China, \{cqxu@bupt.edu.cn\}}

\author{Hongke Zhang}
\address{Beijing Jiaotong University, Beijing, China, \{hkzhang@bjtu.edu.cn\}}

\begin{abstract}
Let $G=(V,E)$ and $H$ be two graphs. Packing problem is to find in $G$ the largest number of independent subgraphs each of which is isomorphic to $H$. Let $U\subset{V}$. If the graph $G-U$ has no subgraph isomorphic to $H$, $U$ is a cover of $G$. Covering problem is to find the smallest set $U$. The vertex-disjoint tree packing was not sufficiently discussed in literature but has its applications in data encryption and in communication networks such as multi-cast routing protocol design. In this paper, we give the kind of $(k+1)$-connected graph $G'$ into which we can pack independently the subgraphs that are each isomorphic to the $(2^{k+1}-1)$-order perfect binary tree $T_k$. We prove that in $G'$ the largest number of vertex-disjoint subgraphs isomorphic to $T_k$ is equal to the smallest number of vertices that cover all subgraphs isomorphic to $T_k$. Then, we propose that $T_k$ does not have the \emph{Erd\H{o}s-P\'{o}sa} property. We also prove that the $T_k$ packing problem in an arbitrary graph is NP-hard, and propose the distributed approximation algorithms.

\end{abstract}

\begin{keyword}
Packing \sep covering \sep perfect binary tree \sep vertex-disjoint

\end{keyword}

\end{frontmatter}


\section{Introduction}
\label{}
Let $G=(V,E)$ and $H$ be two graphs. Packing $H$ into $G$ is to find vertex-disjoint subgraphs each of which is isomorphic to $H$. Covering every subgraph isomorphic to $H$ in $G$ is to find a set $U\subset{V}$ such that $G-U$ has no subgraph isomorphic to $H$. In $G$ the largest number of vertex-disjoint subgraphs that are each isomorphic to $H$ is related to the smallest number of vertices that cover all subgraphs that are each isomorphic to $H$. Packing and covering problems are dual and manifest some properties of the subgraph $H$. Let $y$ denote the smallest number of vertices of a cover. If in arbitrary graph $G$ we have $y\leq{f(x)}$ where $x$ is the largest number of vertex-disjoint subgraphs that are each isomorphic to $H$, then $H$ has the \emph{Erd\H{o}s-P\'{o}sa} property [1]. Cycle has the \emph{Erd\H{o}s-P\'{o}sa} property and packing cycle has been studied in related work [2][3].

Packing and covering also give rise to some rather difficult problems in computer science. Packing the simple subgraph of a single edge into arbitrary graph is called matching. Matching problem is to find the largest number of independent edges and has been proved a NP-complete problem in computational complexity. Many other packing problems are also NP-complete or NP-hard. Since it is difficult to deal with these problems, packing vertex-disjoint graphs of other sorts such as trees were not sufficiently discussed in literature. However, packing vertex-disjoint trees has its applications in data encryption and in communication networks. The distributed multi-cast routing in the current Internet can be modeled as a tree packing problem. The multi-cast routers (vertices of the graph) can be divided into independent groups (trees) for efficient data transmission and interconnection of networks [4]. Emergence of new networking design [5] also take advantage of the content-distributed groups (trees) to disseminate content and ensure no matter which groups are choose each group contains a fixed source router (covering problem).

In this paper, we are primarily concerned with the perfect binary tree packing problem, which includes the kind of graph into which we can pack the tree, in this kind of graph the relation between the largest number of vertex-disjoint subgraphs isomorphic to the tree and the smallest number of vertices that cover all subgraphs isomorphic to the tree, the computational complexity of such vertex-disjoint tree packing into an arbitrary graph, and the applicable approximation algorithms for these problems. For the perfect binary tree $T_1$ with $|T_1|=3$, we give the 2-connected graph $G$ that has the order of $3r$ ($r\geq{1}$ is a positive integer) and can be constructed from the longest cycle $C$ contained in $G$. We prove that in $G$ the largest number $\alpha$ of vertex-disjoint subgraphs isomorphic to $T_1$ is equal to the smallest number $\beta$ of vertices that cover all subgraphs isomorphic to $T_1$. Cycles have the \emph{Erd\H{o}s-P\'{o}sa} property. But perfect binary trees do not have such property that is also proved in this paper. For the perfect binary tree $T_2$ with $|T_2|=7$, we give the 3-connected graph $G$ which can be generated from a minor of $G$ isomorphic to the complete graph $K^4$. The kind of $(k+1)$-connected graph, generated from $K^{k+2}$ as a minor, also have the largest packing subgraph number equal to the smallest covering vertex number. Karp in [6] proposed that matching problem is NP-complete. Based on this, we prove that perfect binary tree packing problem is NP-hard. Since distributed algorithms such as Dijkstra Algorithm in the current Internet have been widely applied in communication networks, we propose distributed approximation algorithms for packing $T_k$ into an arbitrary graph. A key process in the algorithm is to detect the blocks of a connected graph. The $T_1$ packing algorithm has complexity of $O(n^2)$ and the $T_2$ packing algorithm has complexity of $O(n^4)$.

The remainder of this paper is organized as follows. In section 2, we give the lemmas and theorems about packing $(2^{k+1}-1)$-order perfect binary tree into $(k+1)$-connected graph. In section 3, we prove the computational complexity and propose approximation algorithms for packing $T_k$ into an arbitrary graph.

\section{Packing Perfect Binary Trees}
\label{}
Let $T_{k}$ denote the $(2^{k+1}-1)$-order perfect binary tree.

\subsection{Packing $T_1$ into 2-connected graph}
\newtheorem{lemma}{Lemma}[section]
\begin{lemma}
Let $r\geq{1}$ be a positive integer, $P^{3r}$ be a path of length $3r$ and $C^{3r}$ be a cycle of length $3r$. Then

(i) at most $r$ vertex-disjoint subgraphs isomorphic to $T_1$ can be packed into $P^{3r}$; at least $r$ vertices of $P^{3r}$ suffice to meet all its subgraphs isomorphic to $T_1$;

(ii) at most $r$ vertex-disjoint subgraphs isomorphic to $T_1$ can be packed into $C^{3r}$; at least $r$ vertices of $C^{3r}$ suffice to meet all its subgraphs isomorphic to $T_1$.
\end{lemma}
\begin{proof}
The path $P^{3r}$ can be expressed as a sequence of its vertices, i.e. $P^{3r}=a_{0}a_{1}a_{2}...a_{3r-2}a_{3r-1}a_{3r}$. In the same way $C^{3r}$ is expressed as $C^{3r}=a_{0}a_{1}a_{2}...a_{3r-2}a_{3r-1}a_{0}$. We choose the three consecutive points $a_{3i}$, $a_{3i+1}$ and $a_{3i+2}$ where $0\leq{i}\leq{r-1}$ to be a group. Then we can divide $P^{3r}$ or $C^{3r}$ into $r$ groups. In order to pack as many vertex-disjoint subgraphs isomorphic to $T_1$ as possible, we choose the subgraph $S_{i}=a_{3i}a_{3i+1}a_{3i+2}$ and get total $r$ such subgraphs.

The covering problem also needs deliberate selection of a vertex set. Due to the symmetric property of the path and cycle, we just have to decide the distance between two adjacent cover points. We choose the set {${a_{3i+1}}$} as a cover. The vertex $a_{3i+1}$ covers three subgraphs isomorphic to $T_1$: $a_{3i-1}a_{3i}a_{3i+1}$, $a_{3i}a_{3i+1}a_{3i+2}$ and $a_{3i+1}a_{3i+2}a_{3i+3}$. So, the cover set contains $r$ vertices.
\end{proof}

\newtheorem{theorem}[lemma]{Theorem}
\begin{theorem}
Let a 2-connected graph $G$ contains a cycle $C^{3r}$ of length $3r$, which is the longest among all its cycles. Let the length of $G$-path be three or a multiple of three and the ends of $G$-path be in the vertex set covering all $G$'s subgraphs isomorphic to $T_1$. Then in $G$ the largest number $\alpha$ of vertex-disjoint subgraphs isomorphic to $T_1$ is equal to the smallest number $\beta$ of vertices that cover all subgraphs isomorphic to $T_1$.
\end{theorem}
\begin{proof}
According to the property of 2-connected graphs [2], We can reconstruct the 2-connected graph $G$ in this way: (i) to find in $G$ a cycle $C$; (ii)to add $C$-paths to $C$ to form a graph $H$; (iii) to add $H$-paths successively to $H$ until we get $G$. Following this way, we can construct the 2-connected graph that meets the requirements in the theorem. We need (i)find the longest cycle $C$ whose length is $3r$; (ii)find each $C$-path $P$ whose ends are in the covering vertex set $U$ and get $G'=C\cup{P}$; (iii)find possible $G'$-paths whose ends are also in $U$. $C^{3r}$ is expressed as the vertex sequence $C^{3r}=a_{0}a_{1}a_{2}...a_{3r-2}a_{3r-1}a_{0}$. Then we consider four cases:

Case 1: Two adjacent vertices of $C^{3r}$ can not be ends of a $C^{3r}$-path, because lemma 2.1 proves that the vertices $a_{3i}$ and $a_{3i+1}$ are not both in the covering set of $C^{3r}$. According to lemma 2.1, $a_{3i}$ and $a_{3i+2}$ also have no $C^{3r}$-path between them. Therefore this case can be reduced to the lemma 2.1(ii) and we have $\alpha=\beta=r$.

Case 2: The vertices $a_{3i}$ and $a_{3(i+1)}$ can be the ends of a $C^{3r}$-path whose length is exactly three. However, in this case we can pack no independent subgraph isomorphic to $T_1$, and the path contains no vertex in $U$. So we have $\alpha=\beta=r$.

Case 3: The vertices $a_{3i}$ and $a_{3(i+2)}$ can be the ends of a $C^{3r}$-path whose length are either 3 or 6. If the length of the path equals to 3, then no independent subgraph isomorphic to $T_1$ can be packed and we have $\alpha=\beta=r$. If the length of the path equals to 6, we have $\alpha=\beta=r+1$. If there are $q$ independent $C^{3r}$-paths between $a_{3i}$ and $a_{3(i+1)}$ each of which has a length of 6, then we have $\alpha=\beta=r+q$. The only trouble in this case is the situation in which $a_{3i}$ and $a_{3(i+2)}$ are the ends of a $C^{3r}$-path $P_1$, and $a_{3(i+1)}$ and $a_{3(i+3)}$ are the ends of another $C^{3r}$-path $P_2$. But this trouble can not happen because of the requirement of the longest cycle $C^{3r}$. In this situation the distance between $a_{3i}$ and $a_{3(i+3)}$ on the cycle is 9, but we can find a path $P$ between $a_{3i}$ and $a_{3(i+3)}$ as $P=P_{1}a_{3(i+2)}a_{3(i+2)-1}a_{3(i+2)-2}a_{3(i+1)}P_{2}$ whose length $\|P\|=15>9$. This contradicts the the requirement of the longest cycle.

Case 4: The vertices $a_{3i}$ and $a_{3(i+f)}$ can be the ends of a $C^{3r}$-path whose length is $3l$ where $l\leq{f}$. We still have $\alpha=\beta$ in this case. We may consider whether we can add a path with one end in the $C^{3r}$-path and another end in $C^{3r}$. This path, if we can add it, does not have the length more than 3 or else it will contradict the longest cycle. Hence, such path does not enlarge the number $\alpha$ or $\beta$.
\end{proof}

\newtheorem{theorem1}[lemma]{Theorem}
\begin{theorem1}
The perfect binary tree does not have the Erd\H{o}s-P\'{o}sa property.
\end{theorem1}
\begin{proof}
We illustrate this with the opposing example in 2-connected graph. Theorem 2.2 requires the ends of a $C^{3r}$-path to belong to the covering vertex set. We consider that the ends of a $C^{3r}$-path are not in the cover. Let $\alpha$ denote the largest number of vertex-disjoint subgraphs isomorphic to $T_1$ and $\beta$ denote the smallest number of vertices that cover all subgraphs isomorphic to $T_1$. To construct a 2-connected graph, we choose $a_{3i+1}$ and $a_{3i+4}$ as the ends of $h$ independent $C^{3r}$-paths whose length are all three. Both $a_{3i+1}$ and $a_{3i+4}$ are not in the covering vertex set. Then we have $\alpha=r$ and $\beta=r+h$. Obviously, $\beta$ can not be bounded by $\alpha$ since $h$ can be any positive integer.
\end{proof}

\subsection{Packing $T_2$ into 3-connected graph}
Let $G$ and $H$ be two graphs. If $H$ can be obtained from $G$ by contracting some edges of $G$, we call $H$ a minor of $G$ and write $G=MH$.

According to the theorem about connectivity [7], any 3-connected graph can be finally changed into $K^4$ by a series of contraction. In other words, any 3-connected graph can be generated from $K^4$. In order to find the kind of 3-connected graph into which we can pack $T_2$ trees independently, we need build the elementary structure that contains a subgraph isomorphic to $T_2$. The following lemma works on this.

\begin{figure}
\includegraphics[width=14cm]{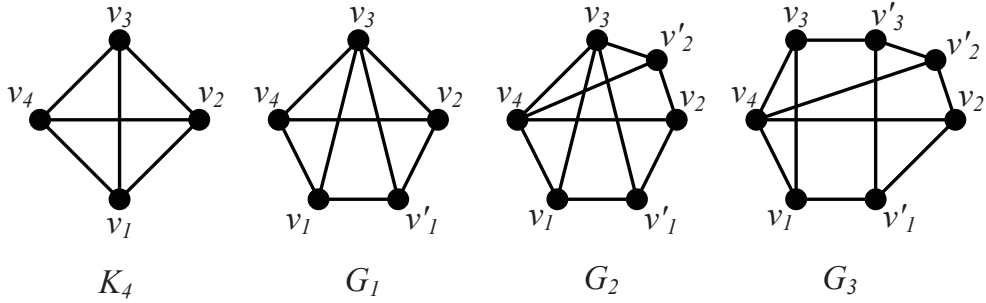}\\
\caption{The elementary structure $G_3$ generated from $K^4$}\label{}
\end{figure}

\newtheorem{lemma1}[lemma]{Lemma}
\begin{lemma1}
The simplest 3-connected graph that contains a subgraph isomorphic to $T_2$ can be obtained from $K^4$ by a series of anti-contraction.
\end{lemma1}
\begin{proof}
Let $\{v_1, v_2, v_3, v_4\}$ be the vertex set of $K^4$. The anti-contraction is to divide one vertex into two to form a new graph whose minor is still $K^4$. We do three times of anti-contraction as follows:

(i) we divide $v_1$ into two vertices $v_1$ and $v'_1$. To keep $d(v_1)=d(v'_1)=3$, we add the new edges $\{v_1v'_1\}$ and $\{v'_1v_3\}$. We call this new graph $G_1$ and have $G_1=MK^4$.

(ii) we divide $v_2$ into two vertices $v_2$ and $v'_2$. To keep $d(v_2)=d(v'_2)=3$, we add the new edges $\{v_2v'_2\}$ and $\{v'_2v_4\}$. We call this new graph $G_2$ and have $G_2=MG_1=MK^4$.

(iii) we divide $v_3$ into two vertices $v_3$ and $v'_3$. To keep $d(v_3)=d(v'_3)=3$, we add the new edges $\{v_3v'_3\}$, $\{v'_1v'_3\}$. We call this new graph $G_3$ as show in Figure 1 and have $G_3=MG_2=MG_1=MK^4$.

To make the structure as simple as possible, we ensure the number of vertices of degree 3 as many as possible after doing the anti-contraction. In this simplest graph, we can find several subgraphs isomorphic to $T_2$. For example, let $v_1$ be the root, then $v_4$ and $v'_1$ are the offsprings of $v_1$, $v'_3$ and $v'_2$ are offsprings of $v_4$, and $v_3$ and $v_2$ are the offsprings of $v'_1$.
\end{proof}

\newtheorem{lemma2}[lemma]{Lemma}
\begin{lemma2}
Let $G$ be a 3-connected graph with $|G|=7r$ and $r\geq{1}$ be a positive integer. Then $G$ can be generated from $G_3$ so that independent subgraphs isomorphic to $T_2$ can be packed into $G$.
\end{lemma2}
\begin{proof}
$G$ can be generated from $G_3$ after $r$ steps. In step $i$ ($2\leq{i}\leq{r}$), we add a component to graph $H_{i-1}$ and get $H_i$. The component is a graph isomorphic to $G_3$. Let $\{v_1, v_2, v_3, v_4, v_5, v_6, v_7\}$ be the vertex set of a subgraph $X_{i-1}$ contained in $H_{i-1}$. $X_{i-1}$ is isomorphic to $G_3$. Let $\{v'_1, v'_2, v'_3, v'_4, v'_5, v'_6, v'_7\}$ be the vertex of $X_{i}$ isomorphic to $G_3$. Then we have $H'_{i}=H_{i-1}\cup{X_{i}}$. We add new edge set $E'=\{\{v_1v'_1\}, \{v_2v'_2\}, \{v_3v'_3\}, \{v_4v'_4\}, \{v_5v'_5\}, \{v_6v'_6\}, \{v_7v'_7\}\}$ to $H'_{i}$ and get $H_{i}=H'_{i}\cup{E'}$.
\end{proof}

\begin{figure}
\centerline{\includegraphics[width=9cm]{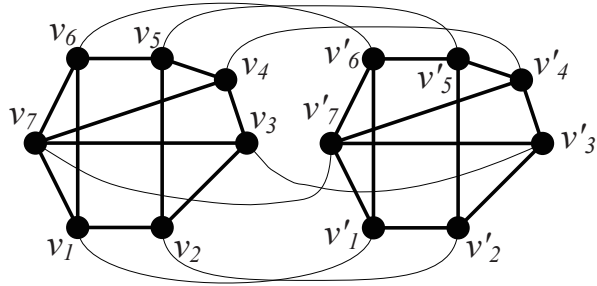}\\}
\caption{The kind of 3-connected graph generated from $G_3$}\label{}
\end{figure}

\newtheorem{theorem2}[lemma]{Theorem}
\begin{theorem2}
In the 3-connected graph $G=H_r$, the largest number $\alpha$ of vertex-disjoint subgraphs isomorphic to $T_2$ is equal to the smallest number $\beta$ of vertices that cover all subgraphs isomorphic to $T_2$.
\end{theorem2}
\begin{proof}
Consider the case in the Figure 2. According to lemma 2.4 and lemma 2.5, we can pack two vertex-disjoint perfect binary trees $T=(V, E)$ and $T'=(V', E')$ into the graph where

$V(T)=\{v_1, v_2, v_3, v_4, v_5, v_6, v_7\}$

$E(T)=\{\{v_1v_7\}, \{v_1v_2\}, \{v_7v_6\}, \{v_7v_4\}, \{v_2v_5\}, \{v_2v_3\}\}$

$V'(T)=\{v'_1, v'_2, v'_3, v'_4, v'_5, v'_6, v'_7\}$

$E'(T)=\{\{v'_1v'_7\}, \{v'_1v'_2\}, \{v'_7v'_6\}, \{v'_7v'_4\}, \{v'_2v'_5\}, \{v'_2v'_3\}\}$

In Figure 2 all the subgraphs isomorphic to $T_2$ come from two parts: (i)the subgraphs isomorphic to $G_3$ contains several subgraphs isomorphic to $T_2$; (ii)because the edge set

$\{\{v_1v'_1\}, \{v_2v'_2\}, \{v_3v'_3\}, \{v_4v'_4\}, \{v_5v'_5\}, \{v_6v'_6\}, \{v_7v'_7\}\}$ exists in the graph, the tree that we pack may contain vertices from different subgraphs isomorphic to $G_3$. For example, the tree $T''$ could be

$V(T'')=\{v_1, v_2, v_3, v_4, v_5, v'_6, v_7\}$

$E(T'')=\{\{v_1v_7\}, \{v_1v_2\}, \{v_7v_6'\}, \{v_7v_4\}, \{v_2v_5\}, \{v_2v_3\}\}$.

These trees make it difficult to decide the number $\beta$. But there are two special vertices in the graph. A 7-order perfect binary tree has three levels. No matter which subgraph is packed, the vertex $v_7$ or $v'_7$ must be at the second level (between the root level and leaf level) because of their degree $d(v_7)=d(v'_7)=5$ larger than other vertices. If we choose $v_7$ and $v'_7$ as the covering vertices, they can cover all the subgraphs isomorphic to $T_2$. Hence, we have $\alpha=\beta$.
\end{proof}

\subsection{Packing $T_k$ into $(k+1)$-connected graph}
\newtheorem{proposition}[lemma]{Proposition}
\begin{proposition}
In the 3-connected graph $G=H_r$, the largest number $\alpha$ of vertex-disjoint subgraphs isomorphic to $T_2$ is equal to the smallest number $\beta$ of vertices that cover all subgraphs isomorphic to $T_2$.
\end{proposition}
\begin{proof}
The $(k+1)$-connected graph that contains a subgraph isomorphic to $T_k$ can be generated from $K^{k+2}$ by doing the anti-contraction successively. Suppose we get the elementary structure $H$, and we add the graph $H'$ isomorphic to $H$ as a component to the new graph $H\cup{H'}$. Then we can add a edge between the corresponding vertices of the two components. The number of edges we add is $2^{k+1}-1$ and ensure that new graph is still $(k+1)$-connected. In this new graph, we can also find two vertices respectively in each component. The two vertices have the degree larger than other vertices. They must be at the second level of a perfect binary tree. So the two vertices can cover all the subgraphs isomorphic to $T_k$.
\end{proof}



\section{Computational Complexity and Approximation Algorithms}
\label{}

\newtheorem{proposition1}{Proposition}[section]
\begin{proposition1}
Packing $T_1$ into an arbitrary graph is a NP-hard problem.
\end{proposition1}
\begin{proof}
According to [6], matching problem is NP-complete. Matching problem is to find largest number of vertex-disjoint edges in an arbitrary graph. Suppose we have an algorithm $L$ to deal with the problem on Packing $T_1$ into an arbitrary graph  $G$. We change $G$ into its line graph $G'$ whose one vertex represents a edge of $G$ and one edge represent that two edges are incident with a same vertex of $G$. By using the algorithm $L$, if we find in $G$ a independent subgraph isomorphic to $T_1$, then we can find an independent edge in $G'$. Therefore, we can use $L$ to deal with the matching problem, and according to [8][9], it means that Packing $T_1$ into an arbitrary graph is harder than the matching problem.
\end{proof}

\newtheorem{proposition2}[proposition1]{Proposition}
\begin{proposition2}
Packing $T_2$ into an arbitrary graph is a NP-hard problem.
\end{proposition2}
\begin{proof}
Let $L$ be the algorithm to deal with the problem on Packing $T_1$ into an arbitrary graph. Suppose we have an algorithm $L'$ to deal with the problem on Packing $T_2$ into an arbitrary graph. We regard each structure of one vertex and its two incident edges of $G$ as a vertex of $G'$. An edge of $G'$ represents that one edge of the structure in $G$ is incident with the vertex of another structure in $G$. By using the algorithm $L'$, if we find in $G$ a independent subgraph isomorphic to $T_2$, then we can find in $G'$ an independent subgraphs isomorphic to $T_1$. Therefore we can use $L'$ to deal with the problem on Packing $T_1$ into an arbitrary graph and $L'$ is harder than $L$. Since $L$ is NP-complete, $L'$ is NP-hard.
\end{proof}

In order to apply the packing in the communication network, we design distributed algorithms. A vertex of a graph represents a network node such as a router in the Internet. The edge between two vertices represents the connection or link between two nodes. Each node maintains a routing table.

According to properties of connectivity of arbitrary graph [2], the graph can be constructed with several blocks and the single path between every two blocks. Thus, before packing we need first detect all the blocks of the graph. Let $G=(V, E)$ be an arbitrary graph with $|G|=n, \|G\|=m$ and $v\in{V}$. Let $Tab(v)$ be the routing table of $v$ and $d(v)$ be the degree of $v$. $Tab(v)$ recorded at most $d(v)$ sequences of vertices each of which represent a independent path to other vertices. Let the set $\{v'_1, v'_2, ..., v'_{d(v)} \}$ be the neighbors of $v$, and $\{p'_1, p'_2, ..., p'_{d(v)} \}$ be the possible independent paths recorded in $Tab(v)$. Let $\{B_1, B_2, ..., B_n\}$ be the possible blocks in graph $G$, $P_qs$ be the path between $B_q$ and $B_s$, and vertices $v_i$ and $v_j$ be any two different vertices of $G$. The block detection algorithm is as below:

\newtheorem{algorithm}[proposition1]{Algorithm}
\begin{algorithm}
\end{algorithm}
\textbf{Input:} An arbitrary graph $G=(V, E)$ with $|G|=n, \|G\|=m$.

\textbf{Output:} The blocks $\{B_1, B_2, ..., B_r, r\leq{n}\}$ and $P_{rs}$ between $B_q$ and $B_s$.
\begin{enumerate}
  \item Every vertex $v$ experiences at most $n$ steps to record the possible independent paths $\{p'_1, p'_2, ..., p'_{d(v)} \}$ in its routing table. In each step, there are two actions: refreshing and broadcasting. During step 1, $v$ records $d(v)$ paths to other vertices at most 1-hop distance away, and broadcasts $Tab(v)$ to $\{v'_1, v'_2, ..., v'_{d(v)} \}$. During step 2, $v$ receives $\{Tab(v'_1), Tab(v'_2), ..., Tab(v'_{d(v)}) \}$, refreshes $Tab(v)$ to be $d(v)$ paths to other vertices at most 2-hop distance away, and broadcast new $Tab(v)$ to to $\{v'_1, v'_2, ..., v'_{d(v)} \}$. During step $k$ ($3\leq{k}\leq{n}$), $v$ receives new tables from its neighbors, refreshes $Tab(v)$ to be $l$ ($l\leq{d(v)}$) paths to other vertices at most $k$-hop distance away by deleting the path that has more than one joint vertices with another path (the only joint vertex, if it exists, must be a end of a recorded path).
  \item Compute the number of independent paths between two different vertices $v_i$ and $v_j$: search $Tab(v_i)$ and find if $v_j$ appears in $Tab(v_i)$ and how many times it appears. There are $l$ ($l\leq{d(v_i)}$) paths in $Tab(v_i)$. After $l$ times of search, the number $x_{ij}$ of independent paths can be obtained.
  \item For $i\leq{n-1}$ and $i<j\leq{n}$, if $x_ij=0$, undo; else if $x_{ij}=1$ and $d(v_i)=2$ and $d(v_j)=1$, undo; else if $x_{ij}=1$ and $d(v_i)>2$ and $d(v_j)\leq{2}$, then put $v_i$ into $B_i$; else if $x_{ij}=1$ and $d(v_i)>2$ and $d(v_j)>2$, then put $v_i$ into $B_i$ and put $v_j$ into $B_j$; else if $x_{ij}>{1}$, then put both $v_i$ and $v_j$ into $B_i$; else if $x_{ij}=1$ and $d(v_i)=d(v_j)=2$, then $v_i$ and $v_j$ are both on the path between $B_q$ and $B_s$ and put $v_i$ and $v_j$ into $P_{qs}$.
   \item End.

\end{enumerate}

Complexity analysis: the algorithmic executing time is consumed by two processes: (i)each vertex experience at most $n$ steps to get its final routing table and the whole network experience at most $n$ steps because the vertices work independently at same time; (ii)for every vertex couple of $v_i$ and $v_j$, $v_i$ has to do at most $d(v_i)$ ($d(v_i)\leq{\Delta{(G)}}, \Delta{(G)}=\max\{d(v)|v\in{V}\}$) times of search and there are total $\frac{n(n-1)}{2}$ couples. So, we have $f(n)\leq{n+\Delta(G)\cdot{\frac{n(n-1)}{2}}}$ and the algorithm complexity is $O(n^2)$.

\newtheorem{algorithm1}[proposition1]{Algorithm}
\begin{algorithm1}
\end{algorithm1}
\textbf{Input:} An arbitrary graph $G=(V, E)$ with $|G|=n, \|G\|=m$.

\textbf{Output:} Vertex-disjoint subgraphs isomorphic to $T_1$.

\begin{enumerate}
  \item Run Algorithm 3.3 to do the block detection.
  \item Find the longest cycle in each block: (i)choose $f$ ($1\leq{f}\leq{n}$, because the worst case is that a block contains all $n$ vertices) vertices as the beginning vertices to form a cycle; (ii)from the beginning vertex $v_i$, the longest cycle $C_{v_i}$ can be obtained by using the distributed algorithm to find longest-distance between any two vertices; (iii)cancel the vertices on $C_{v_i}$ whose degree are equal to 2; choose next beginning vertex $v_j$ from the remaining vertices on $C_{v_i}$ and do the longest cycle $C_{v_j}$search again; (iv)if $\|C_{v_i}\|<\|C_{v_j}\|$, then cancel $v_i$, or else cancel $v_j$.
  \item Pack $T_1$ into the longest cycle $C_k$ in block $B_k$. Let $P$ be a $C_k$-path. If $\|p\|>3$, then Pack $T_1$ into $P$; else undo. If $\|P_{rs}\|>3$, then Pack $T_1$ into $P_{rs}$; else undo.
  \item End.
\end{enumerate}

Complexity analysis: Algorithm 3.3 has the complexity of $O(n^2)$. It takes the complexity of $O(n^2)$ to do the longest cycle search. In the worst case, a block contains all $n$ vertices each of which has a degree larger than 2. We need $\frac{n(n-1)}{2}$ times of cycle length comparison. So, Algorithm 3.4 has the complexity of $O(n^2)$.

\newtheorem{algorithm2}[proposition1]{Algorithm}
\begin{algorithm2}
\end{algorithm2}
\textbf{Input:} An arbitrary graph $G=(V, E)$ with $|G|=n, \|G\|=m$.

\textbf{Output:} Vertex-disjoint subgraphs isomorphic to $T_2$.

\begin{enumerate}
  \item Run Algorithm 3.3 to do the block detection.
  \item In each block, find vertex-disjoint subgraphs isomorphic to $K^4$ and judge whether the subgraphs isomorphic to $G_3$ can be constructed by searching the neighbors of the subgraphs isomorphic to $K^4$.
  \item Pack $T_2$ into the subgraphs isomorphic to $G_3$.
  \item End.
\end{enumerate}

Complexity analysis: the algorithmic executing time is mainly consumed in finding vertex-disjoint subgraphs isomorphic to $K^4$. So, Algorithm 3.5 has the complexity of $O(n^4)$.

\begin{flushleft}
    \textbf{Acknowledgment}
\end{flushleft}

This work is supported by the National High-Tech Research and Development Program of China (863) under Grant No. 2011AA010701, and partially supported by the National Basic Research Program of China (973 Program) under Grant 2013CB329100 and 2013CB329102.








\end{document}